\documentclass[J,12pt]{amsart}
\usepackage{epsfig,epsf,amsmath,amssymb}
\newtheorem{theorem}{Theorem}[section]
\newtheorem{corollary}[theorem]{Corollary}
\newtheorem{definition}{Definition}
\newtheorem{example}[theorem]{Example}
\newtheorem{proposition}[theorem]{Proposition}
\newtheorem{remark}[theorem]{Remark}
\newtheorem{lemma}[theorem]{Lemma}

\begin{document}

\author[Sang-Eon Han]{Sang-Eon~Han}

\title[]{Algebraic structures on digital objects}
\thanks{2010 {\it Mathematics Subject Classification}: 22A05,22A10,54C08, 54H11,68U10}\keywords{Khalimsky topology, $DT$-$k$-ring, $DT$-$k$-filed, pseudo-$DT$-$k$-ring, pseudo-$DT$-$k$-filed, $DT$-$k$-isomorphism}

\maketitle

\begin{center}
		{\it Department of Mathematics Education,
		Jeonbuk National University, 567 Baekje-daero, deokjin-gu, Jeonju-si, Jeollabuk-do 54896, Republic of Korea.\\
	E-mail address: sehan@jbnu.ac.kr	} 
	
\end{center}

\begin{abstract}
 The paper aims to introduce a digital-topological ($DT$-, for brevity) $k$-ring and a $DT$-$k$-field.
 They are indeed endowed with both a digital image (or digital object) $(X, k)$
 and a ring structure or a field structure $(X, \ast_1, \star)$, where $X \subset {\mathbb Z}^n$ and the $k$-adjacency is the digital $k$-connectivity of  ${\mathbb Z}^n$.
 Besides, some properties of them are investigated.
 The ring $(SC_k^{n,l}, \ast_1, \star)$ is proved to be isomorphic to the ring $({\mathbb Z}_l, +, \cdot)$,
  where $SC_k^{n, l}$ is a simple $k$-cycle with $l$ elements in ${\mathbb Z}^n$, $n\in {\mathbb N}\setminus \{1\}$, and ${\mathbb N}$ is the set of natural numbers.
  However, $(SC_k^{n,l}, \ast_1, \star)$ is proved not to be a $DT$-$k$-ring.
     Meanwhile, we prove that for $l \in \mathcal{P} \setminus \{2,3\}$, while $(SC_k^{n, l}, \ast_1, \star)$ is a field, it cannot be a $DT$-$k$-field,
    where $\mathcal{P}$ indicates the set of prime numbers.
  Besides, the paper proves that the field $(X:=\{-1, 0, 1\}, \ast_1, \star)$ is a $DT$-$2$-field derived from the digital image $(X, 2)$ and the field $(X:=\{-1, 0, 1\}, \ast_1, \star)$,
   and further,
  $(Y:=\{0, 1\}, \ast_1, \star)$ is also a $DT$-$2$-field derived from the digital image $(Y, 2)$ and the field $(Y:=\{0, 1\}, \ast_1, \star)$.
  Furthermore, we prove that the $n$-dimensional digital cube $(X^n,\ast_1, \star)$ is a $DT$-$2n$-field and  $(Y^n,\ast_1, \star)$ is also a $DT$-$2n$-field.
  In addition,  given any digital cube with $2n$-adjacency on ${\mathbb Z}^n$, e.g., 
  $([a, a+2]_{\mathbb Z})^n$, after transforming it onto $([-1,1]_{\mathbb Z}^n, 2n)$ by using a $2n$-isomorphism (or a digital topological $2n$-embedding), the transformed digital image can also be a $DT$-$2n$-field.   
  Finally, the paper introduces the notions of a pseudo $DT$-$k$-ring and a pseudo $DT$-$k$-field which are weaker than those of a $DT$-$k$-ring and a $DT$-$k$-field, respectively.
 \end{abstract}

\thanks {AMS Classification: 22A05, 22A10, 54C08, 54H11, 68U05, 68U10\\
	Keywords: Khalimsky topology, $DT$-$k$-ring,  $DT$-$k$-field, pseudo $DT$-$k$-ring, pseudo $DT$-$k$-field, $DT$-$k$-group isomorphism.}


\maketitle

\newpage

 \section{\bf{Introduction}}\label{s1}

 In digital and applied topology associated with computer science, we have the following interesting query:
 What kinds of both algebraic- and topological-jointed structures on digital objects can be considered\,?
 In the case that we formulate the structures,
 they can facilitate many studies of digital objects from the viewpoints of applied topology as well as digital topology, digital geometry, computer science, and so on.
 To address this topic, motivated by the typical topological group \cite{A1,T1},
  a recent paper \cite{H7} proposed the so-called digital-topological ($DT$-, for brevity) $k$-group endowed with both a digital image $(X,k), X \subset {\mathbb Z}^n$ (see Section 2) and a group $(X, \ast_1)$. 
  To do this work, the paper \cite{H7} used a special kind of adjacency relation on $X^2:=X\times X$ which was motivated by the graph adjacency of a Cartesian product in \cite{B1,H3,H4,H9,Sa1}. 	
 Besides, it investigated many properties of a $DT$-$k$-group structure.
 The recent paper \cite{H8} also investigated several types of $DT$-$k$-subgroup structures.
 Furthermore, it proved that each of the $n$-dimensional Khalimsky topological space and the Marcus-Wyse topological plane cannot be a typical topological group. \\
  
  Based on a $DT$-$k$-group structure, after establishing  both an abelian group $(X, \ast_1)$ and a commutative semigroup  $(X, \star)$ in which the operation $\star$ is associative and distributive over the operation  $\ast_1$,
  we first establish a commutative  $DT$-$k$-ring (see Section 4).
  Furthermore,
  the paper will initially propose a $DT$-$k$-field (see Section 5) structure based on the digital image $(X, k)$, abelian group $(X, \ast_1)$, and a commutative semigroup $(X, \star)$ (see Definition 7). In the paper we will denote it by $(X,k,\ast_1,\star)$.
  Indeed, while a $DT$-$k$-field implies a commutative $DT$-$k$-ring (compare Definitions 5 and 7), the converse does not hold. \\

   A $G_{k^\ast}$-adjacency of a digital product $X^2$ will be used for constructing a $DT$-$k$-ring and a $DT$-$k$-field, (see Definition 1 in the paper).
  Indeed, the  $G_{k^\ast}$-adjacency was proved to support the $(G_{k^\ast}, k)$-continuity of a multiplication
 $\alpha: (X^2, G_{k^\ast}) \to (X, k)$ \cite{H7} (see Definition 3). We note that the  
 $G_{k^\ast}$-adjacency  need not be equal (or equivalent) to one of the typical $k$-adjacency relations of ${\mathbb Z}^{2n}$ \cite{H7} 
 (in detail, see Section 3 in the present paper), where $\{k:=k(t, 2n)\,\vert \, t \in [1, 2n]_{\mathbb Z}\}$ as in (2.1), (see also Lemma 3.1 in the present paper).
  Besides, it is also essential to formulating both a  $DT$-$k$-ring and a $DT$-$k$-field.
     In addition, for $\{s, t\} \subset \mathbb Z$ with $s \lneq t$, the notation $[s, t]_{\mathbb Z}$ is often used to indicate the set 
      $\{x \in {\mathbb Z} \,\vert \, s\leq x \leq t\}$.\\

Since each of the $DT$-$k$-ring and  $DT$-$k$-field structure is richer than a $DT$-$k$-group structure, we may raise the following query.\\
 $\bullet$ How to establish a $DT$-$k$-ring on  a digital object $(X, k)$\,?\\
    After developing the notion of a $DT$-$k$-ring, the following queries have also arisen.\\ Hereinafter, let $SC_k^{n,l}$ be a simple closed $k$-curve with $l$ elements in ${\mathbb Z}^n$ (see Section 2 in detail).\\
 $\bullet$ Do the $DT$-$2n$-group $({\mathbb Z}^n, 2n, \ast_1)$ and $DT$-$k$-group $(SC_k^{n,l}, \ast_1)$ introduced in \cite{H7} become a $DT$-$2n$-ring and a $DT$-$k$-ring, respectively\,?\\
       $\bullet$ What are some properties of a $DT$-$k$-ring\,?\\
      $\bullet$ How to establish a $DT$-$k$-field structure derived from both a digital image $(X, k)$ and a field $(X, \ast_1, \star)$\,?\\
     $\bullet$ Is there a  $DT$-$k$-field structure on $SC_k^{n, l}$\,?\\
      $\bullet$ What kinds of  $DT$-$k$-fields exist\,?\\
      $\bullet$ What kinds of structures can be placed between a $DT$-$k$-group and a $DT$-$k$-ring and further,  a $DT$-$k$-ring and a $DT$-$k$-field\,?\\
    After developing several new notions, we will address these topics.\\

 The frame of the paper is shown as follows:
Section 2 deals with some terminology needed for the study in the paper.
In Section 3, given a digital image $(X, k), X \subset {\mathbb Z}^n$, some properties of the $G_{k^\ast}$-adjacency of the digital product $X^2$
and  the $(G_{k^\ast}, k)$-continuity of a map from $(X^2, G_{k^\ast})$ to $(X, k)$  are mentioned.
Section 4 develops a ring structure on $SC_k^{n, l}$.
Besides, it proposes the notion of  a $DT$-$k$-ring and proves that the $DT$-$2n$-group $({\mathbb Z}^n, 2n, \ast_1)$ is not a  $DT$-$2n$-ring  and the $DT$-$k$-group $(SC_k^{n,l}, \ast_1)$ is not a $DT$-$k$-ring either.
Section 5 proposes a field structure on $SC_k^{n, l}$, where $l \in \mathcal{P} \setminus \{2,3\}$,
where $ \mathcal{P}$ is the set of prime numbers.
For $l\in {\mathbb N}\setminus \{2,3\}$, the ring $(SC_k^{n,l}, \ast_1, \star)$ is proved to be isomorphic to the typical ring $({\mathbb Z}_l, +, \cdot)$, where ${\mathbb Z}_l={\mathbb Z}/l{\mathbb Z}:=[0, l-1]_{\mathbb Z}$
      and the operations $\lq\lq$$+"$ and $\lq\lq$$\cdot"$ means the operations $\lq\lq$$+(mod\, l)"$ and $\lq\lq$$\cdot(mod\, l)"$ in the paper, respectively.
    After that, we will introduce a $DT$-$k$-field and prove that while $(SC_k^{n, l}, \ast_1, \star)$ is a field, it is neither a  $DT$-$k$-ring nor a  $DT$-$k$-field, where $l \in \mathcal{P} \setminus \{2,3\}$.
           However, the paper shows that the field $(X:=\{-1, 0, 1\}, \ast_1, \star)$ is a $DT$-$2$-field derived from the digital image $(X, 2)$ and the field $(X:=\{-1, 0, 1\}, \ast_1, \star)$,
   and
  $(Y:=\{0, 1\}, \ast_1, \star)$ is also a $DT$-$2$-field endowed with the digital image $(Y, 2)$ and the field $(Y:=\{0, 1\}, \ast_1, \star)$.
  Furthermore, the $n$-dimensional digital cube $(X^n,\ast_1, \star)$ is a $DT$-$2n$-field and  $(Y^n,\ast_1, \star)$ is also a $DT$-$2n$-field.
  In addition,  given any digital cube with $2n$-adjacency on ${\mathbb Z}^n$, using a $2n$-isomorphism (or digital topological $2n$-embedding \cite{H6}), we can transform it as a $DT$-$2n$-field.     
      Section 6 introduces the notions of a pseudo $DT$-$k$-group, a pseudo $DT$-$k$-ring, and a pseudo $DT$-$k$-field. Section 7  refers to some remarks and  further work.\\
      
   In the paper, only nonempty and $k$-connected digital images $(X, k)$ are considered.
    Besides, given a set $X$, we will use the notation $\sharp X$ to indicate the cardinality of the given set $X$ and the notation $\lq\lq$$:=$" will be used for introducing a new term.

\section{\bf{Some terminology for studying digital images}}\label{s2}

In order to formulate the structure of a $DT$-$k$-ring (resp. $DT$-$k$-field) on a digital object $(X, k)$ with a ring (resp. field) $(X, \ast_1, \star)$,
the  digital $k$-connectivities of  ${\mathbb Z}^n, n \in {\mathbb N}$, referred to in (2.1) below, are strongly required (see Sections 3, 4, and 5).
The $k$-adjacencies of ${\mathbb Z}^n, n \in [1,3]_{\mathbb Z}$, were initially developed by Rosenfeld \cite{R1,R2}.
 As a generalization of them, the papers \cite{H1,H5} establish the following: for $t \in [1,n]_{\mathbb Z}$,
 the distinct points in  ${\mathbb Z}^n$
 $$p = (p_i)_{i \in [1,n]_{\mathbb Z}} \,\, \text{and}\,\, q = (q_i)_{i \in [1,n]_{\mathbb Z}}$$
 are said to be $k(t, n)$-adjacent if at most $t$ of their coordinates  differ by $\pm1$ and the others coincide.
  According to this statement, the $k(t, n)$-adjacency relations of ${\mathbb Z}^n, n \in {\mathbb N}$, were formulated in \cite{H1,H5}, as follows:
    $$k:=k(t, n)=\sum_{i=1}^{t} 2^{i}C_{i}^n, \text{where}\,\, C_i ^n= {n!\over (n-i)!\ i!}. \eqno(2.1)$$
    For instance,
$$(n, t, k) \in \left \{
\aligned
& (4, 4, 80), (4, 3, 64), (4, 2, 32), (4, 1, 8),\\
& (5, 5, 242),  (5, 4, 210),  (5, 3, 130), (5, 2, 50), (5, 1, 10),  \,\text{and}\\
&  (6, 6, 728), (6, 5, 664), (6, 4, 472),  (6, 3, 232),  (6, 2, 72), (6, 1, 12).
\endaligned
\right\} \eqno(2.2)
$$

For a set $X \subset {\mathbb Z}^n, n \in {\mathbb N}$, with one of the $k$-adjacency relations of (2.1), hereinafter, we call $(X, k)$ a digital image (or digital object).

Given $(X, k)$, it is said to be $k$-path connected (or $k$-connected)
if for any distinct points $x, y \in X$, there exists a finite sequence $(x_i)_{i \in [0, n]_{\mathbb Z}} \subset X$ such that $x_0=x$, $x_n=y$, and $x_i$ is  $k$-adjacent to $x_{i+1}$, $i \in [0,n-1]_{\mathbb Z}$ \cite{KR1}.
 Also, a digital image  $(X, k)$ with a singleton is assumed to be $k$-connected for any $k$-adjacency.
A simple closed $k$-curve (or simple $k$-cycle, for brevity) with $l$ elements in ${\mathbb Z}^n$, denoted by $SC_k^{n,l}, 4 \leq l \in {\mathbb N}$ \cite{H1,H5,H6,KR1},
was initially defined as a sequence $(x_i)_{i \in [0, l-1]_{\mathbb Z}}$ on ${\mathbb Z}^n$, where
$x_i$ and $x_j$ are $k$-adjacent if and only if $\vert i-j\vert=\pm1(mod\, l)$ \cite{KR1}. The number $l$ of $SC_k^{n,l}$ depends on both the digital $k$-adjacency and the dimension $n$ (see (5) of \cite{H6}).

A digital $k$-neighborhood of the point $x_0$ with radius $1$ \cite{H1,H2} can be represented as follows:\\
$$N_k(x_0, 1)=\{x \in X \,\vert\,\,x\,\,\text{is}\,\,k\text{-adjacent to}\,\,x_0\} \cup \{x_0\}. \eqno(2.3) $$

This $k$-neighborhood will be strongly used for formulating a  $G_{k^\ast}$-adjacency of a digital product in Section 3.
The digital continuity of \cite{R2} can be represented by using a digital $k$-neighborhood in (2.3) as in Proposition 2.1. Recall that the paper \cite{R2} characterized 
the $(k_0, k_1)$-continuity of the map $f: (X,k_0) \to (Y, k_1)$ as follows:
A function $f: (X,k_0) \to (Y, k_1)$ is $(k_0, k_1)$-continuous if and only if  two points $x$ and $x^\prime$ which are $k_0$-adjacent in $(X,k_0)$ implies that either $f(x)=f(x^\prime)$ or $f(x)$ is $k_1$-adjacent to $f(x^\prime)$ (see Theorems 2.1 and 2.4 of \cite{R2}). 
Indeed, it is clear that this approach can be represented by using a digital $k$-neighborhood of (2.3), which can be effective to study $DT$-agebraic structures later (see Sections 3, 4, 5, and 6).

\begin{proposition} \cite{H1,H5,R2}
Assume $X \subset {\mathbb Z}^{n_0}$ and $Y \subset {\mathbb Z}^{n_1}$.
A mapping $f: (X,k_0) \to (Y, k_1)$ is $(k_0, k_1)$-continuous
if and only if for each $x\in X$, $f(N_{k_0}(x, 1))\subset  N_{k_1}(f(x), 1)$.
\end{proposition}

In Proposition 2.1, in the case that  $n_0 = n_1$ and $k_0 = k_1$, we say that it is {\it $k_0$-continuous}.\\

In addition, a $(k_0, k_1)$-homeomorphism \cite{B2} or a $(k_0,k_1)$-isomorphism (or $k$-isomorphism)\cite{H2,HP2} was also used for classifying digital images.
Naively, given objects $(Z, k_0)$ and $(W, k_1)$,  $h :Z \to W$ is said to be a $(k_0, k_1)$-isomorphism if  $h$
is a  $(k_0, k_1)$-continuous bijection and further, $h^{-1}: W \to Z$ is $(k_1, k_0)$-continuous.

\section{\bf The $(G_{k^\ast}, k)$-continuity supporting a $DT$-$k$-ring and $DT$-$k$-field structure}\label{s3}

   Assume  $(X_i, k_i), X_i \subset {\mathbb Z}^{n_i}, i \in \{1, 2\}$. Then 
  the recent paper \cite{H7} introduced the notion of a $G_{k^\ast}$-adjacency relation on the Cartesian product $X_1\times X_2$ to obtain a relation set (or digital space)
  $(X_1\times X_2, G_{k^\ast})$ and the $(G_{k^\ast}, k_i)$-continuities of maps $f:(X_1\times X_2, G_{k^\ast}) \to (X_i, k_i), X_i \subset {\mathbb Z}^{n_i}, i \in \{1, 2\}$.  
 After that, various properties of a $G_{k^\ast}$-adjacency were also investigated in \cite{H7}.
 However, the present paper focuses on developing various algebraic structures based on a $DT$-$k$-group structure such as a $DT$-$k$-ring, $DT$-$k$-fied, and so on.
 Hence as a special case of $(X_1\times X_2, G_{k^\ast})$ in Definition 4.1 of \cite{H7}, we only need the relation set $(X^2, G_{k^\ast})$ induced by 
  $(X, k), X \subset {\mathbb Z}^n$.
    Let us first recall the notion of a $G_{k^\ast}$-adjacency relation on the Cartesian product $X^2:=X \times X \subset {\mathbb Z}^{2n}$ (see Definition 1 below) and the 
  $(G_{k^\ast}, k)$-continuity of a map 
  $$f:(X^2, G_{k^\ast}) \to (X, k)\,\,\text{(see Definition 3 below)}.$$
   This notion will strongly be used for establishing both an abelian $DT$-$k$-group, a commutative $DT$-$k$-ring, and a $DT$-$k$-field in Sections 4 and 5.
As a special case of Definition 4.1 of \cite{H7}, the following can be considered.

\begin{definition}\cite{H7} 
		Assume  $(X, k:=k(t, n)), X \subset {\mathbb Z}^n$.
	Consider distinct points $p:=(x_1, x_2), q:=(x_1^{\prime}, x_2^{\prime})  \in X^2 \subset {\mathbb Z}^{2n}$ satisfying the condition
	$$q \in (N_k(x_1, 1) \times \{x_2\})\cup (\{x_1\} \times N_k(x_2, 1)). \eqno(3.1)$$
	Then we say that $p$ is related to $q$ on $X^2$.			 	
	After that, as a convenience, taking only the $k(t, 2n)$-adjacency of ${\mathbb Z}^{2n}$ and just putting $k^\ast:=k(t, 2n)$, we say that the relation between $p$ and $q$ is $G_{k^\ast}$-adjacent on $X^2$ derived from the given
	$(X, k:=k(t, n))$.
			\end{definition}
		
		As stated in Definition 1, we carefully note that the number $t$ of $k:=k(t,n)$ is exactly equal to the number $t$ in the $k^\ast:=k(t,2n)$ for the $G_{k^\ast}$-adjacency.
	
	Based on Definition 1,  in the case of $X^2={\mathbb Z}^{2n}$, it is clear that a $G_{k^\ast}$-adjacency on $X^2$ is equal to the $k(1,2n)$-adjacency on $X^2$, where 
	$k^\ast=k(1, 2n)$ which is equal to the $2n$-adjacency on $X^2$.
	However,  in the case of $X^2 \neq {\mathbb Z}^{2n}$,  a $G_{k^\ast}$-adjacency on $X^2 \subsetneq {\mathbb Z}^{2n}$ has its own features, i.e.,  a $G_{k^\ast}$-adjacency on $X^2$ need not be equal to the $k^{\ast}:=k(1, 2n)$-adjacency on $X^2$. Besides, it need not be equal to certain $k$-adjacency of (2.1), as follows:
	
	\begin{lemma}
		(1) A $G_{k^\ast}$-adjacency need not be equal to one of the $k$-adjacency relations of ${\mathbb Z}^{2n}$ stated in (2.1) \cite{H7}
		(see Example 3.2 below) because the $G_{k^\ast}$-adjacency on $X^2$ is defined by using the condition (3.1), i.e., given distinct points $p,q$ of Definition 1 are  $G_{k^\ast}$-adjacent if they satisfy only the condition (3.1). \\		
		(2) A $G_{k^\ast}$-adjacency on $X^2 \subset {\mathbb Z}^{2n}$ is broader  than the $k^\ast$-adjacency of $X^2 \subset {\mathbb Z}^{2n}$ of (2.1) \cite{H7}, where 
		$k^\ast=k(t, 2n)$.
		Namely, for two distinct points $p,q \in X^2$, if $p$ is $G_{k^\ast}$-adjacent to $q$, then they are $k^\ast$-adjacent. However, the converse does not hold. 
		\end{lemma}
	\begin{proof} The proofs of (1) and (2) are straightforward from Definition 1.
		\end{proof}
	
	\begin{example} Let $Y=SC_8^{2,6}:=(y_0, y_1, y_2, y_3, y_4, y_5)$, where $y_0=(0,0), y_1=(1,-1), y_2=(2,-1), y_3=(3,0), y_4=(2,1), y_5=(1,1)$ just an example.
		Let us consider the Cartesian product $Y^2$.
		According to the condition of (3.1), we obtain the $G_{k^\ast}$-adjacency on $Y^2$ and finally 
		a digital space $(Y^2, G_{k^\ast})$ \cite{H7,H10}. Namely,			
		the relation $ G_{k^\ast}$ on $X^2$ is non-reflexive and symmetric \cite{H7,H10}.		
		Then we strongly note that  $G_{k^\ast}\neq k^\ast$, where $k^\ast=k(2,4)=32$ (see (2.1)).\\
		To be specific, consider the two points $p:=(y_2,y_2), q:=(y_1,y_1) \in Y^2$
		Then it is clear that while $q \in N_{k(2,4)}(p, 1)= N_{32}(p, 1)$, $q$ is not $G_{k^\ast}$-adjacent to $p$, where $k^{\ast}=k(2,4)$.
				\end{example}
	To characterize the relation $G_{k^\ast}$ On $X^2$ of Definition 1, including several examples in \cite{H7}, we have the following:\\
	
	\begin{remark} A $G_{k^\ast}$-adjacency need not be equal to a $C_{k^\ast}$-adjacency relation on $X^2 \subset {\mathbb Z}^{2n}$ \cite{H7}.\\		
Given a set  $X^2 \subset {\mathbb Z}^{2n}$, let us compare between a $G_{k^\ast}$-adjacency relation on $X^2$ and a $C_{k^\ast}$-adjacency relation on $X^2$.
To be specific, unlike Definition 1, for any point $p:=(x_1,x_2) \in X^2$, a $C_{k^\ast}$-adjacency relation on $X^2 \subset {\mathbb Z}^{2n}$ \cite{H4} can be represented as a relation on $X^2$ satisfying the condition
$$N_{k^\ast}(p,1)= (N_k(x_1, 1) \times \{x_2\})\cup (\{x_1\} \times N_k(x_2, 1)), \eqno(3.2)$$
where $k^\ast=k(t, 2n)$, and the number $t$ of $k(t, 2n)$ is equal to the number $t$ for 
$k(t,n)$ of $(X, k:=k(t,n))$ and finally $C_{k^\ast}=k^\ast$ (see more detail in \cite{H4}).
Concretely, while a $C_{k^\ast}$-adjacency relation on $X^2$ is one of the $k=k(t,2n)$-adjacencies of  ${\mathbb Z}^{2n}$ as stated in (2.1), a $G_{k^\ast}$-adjacency relation on $X^2$ need not be equal to one of the $k=k(t,2n)$-adjacencies of  ${\mathbb Z}^{2n}, t \in [1,2n]_{\mathbb Z}$, i.e.,  a $G_{k^\ast}$-adjacency relation on $X^2$ is a new relation on $X^2 \subset {\mathbb Z}^{2n}$ (see \cite{H7} for more detail).
Hence we note that 	a $C_{k^\ast}$-adjacency relation on $X^2$ is more rigid than a $G_{k^\ast}$-adjacency relation on $X^2$. Besides, given $(X,k), X \subset {\mathbb Z}^n$,  it is well known that  the set $X^2$ need not have a $C_{k^\ast}$-adjacency relation \cite{H4}.
 	\end{remark}
	
	\begin{remark}  (Advantages of the usage of a $\lq\lq$$G_{k^\ast}$-adjacency" of Definition 1) 
		As for the notation $G_{k^\ast}$ using $k^\ast$ in Definition 1, we need to 
		refer to some benefits of taking the notation because the $G_{k^\ast}$-adjacency has much information identifying the relation of (3.1), as follows: \\	
		Given $(X, k:=k(t, n)), X \subset {\mathbb Z}^{n}$, a $G_{k^\ast}$-adjacency on $X^2$ uniquely and always exists. 
		This $G_{k^\ast}$-adjacency is compared to the $C_{k^\ast}$-adjacency of \cite{H7} (see Definition 3.2 of \cite{H7}). 
		Indeed, a $G_{k^\ast}$-adjacency is clearly broader than a $C_{k^\ast}$-adjacency as stated in Sections 3 and 4 of \cite{H7} (see also the compatible $k^\ast$-adjacency in \cite{H3,H4}). In addition, in detail, see Remark 3.1.
		\end{remark}
				
					We now establish the following $G_{k^\ast}$-neighborhood of a given point of  $X^2$.
				
				\begin{definition} \cite{H7}  Given $(X, k:=k(t, n)), X \subset {\mathbb Z}^{n}$, assume the relation set $(X^2, G_{k^\ast})$.
					For a point $p \in X^2$, we define
					$$N_{G_{k^\ast}}(p):=\{q \in X^2\,\vert \, q \,\,\text{is}\,\,G_{k^\ast}\text{-adjacent to}\,\,p\} \eqno(3.3)$$
					and
					$$N_{G_{k^\ast}}(p, 1):=N_{G_{k^\ast}}(p) \cup \{p\}. \eqno(3.4)$$
					Then $N_{G_{k^\ast}}(p, 1)$ is called a $G_{k^\ast}$-neighborhood of $p$.
				\end{definition}
				
				In view of (3.1) and (3.4), for the relation set $(X^2, G_{k^\ast})$ and a point $p:=(x_1, x_2) \in X^2$, we have the following \cite{H7}:
				$$N_{G_{k^\ast}}(p, 1)= (N_{k}(x_1, 1) \times \{x_2\}) \cup  (\{x_1\}\times N_{k}(x_2, 1)). \eqno(3.5)$$
				
		\begin{remark}
		(1) The choice of the subscript $k^{\ast}$ of $G_{k^\ast}$ can be very helpful to characterize the relation $G_{k^\ast}$. To be specific, 
		as mentioned in Example 3.2, Remarks 3.3 and 3.4, for $p \in X^2$, we always obtain 
		$$N_{G_{k^\ast}}(p,1) \subset N_{k^\ast}(p, 1).\eqno (3.6)$$
		And these two sets need not be equal to each other.	\\
		For instance, based on Example 3.2, we have
		$$N_{G_{k^\ast}}(p,1) \subsetneq N_{k^\ast}(p, 1)$$ 
		because for $p=(y_2,y_2) \in SC_8^{2,6}\times SC_8^{2,6}$, we have
		
		$$\left \{
		\aligned & N_{G_{k^\ast}}(p,1)=\{(y_2,y_1), p, (y_2, y_3), (y_3,y_2), (y_1, y_2)\}  \,\,\text{and}\\
		& N_{k^\ast}(p, 1)=N_{G_{k^\ast}}(p,1) \cup \{(y_1, y_1)\},
		\endaligned
		\right\} $$
		which supports the property of (3.6), where $k^\ast=k(2,4)$.\\
		(2) In the case that there exists a  $C_{k^\ast}$-adjacency of  $X^2$ (Definition 3.2 of \cite{H7}), it is a special case of the $G_{k^\ast}$-adjacency of  $X^2$.
	\end{remark}
To support Definition 1 and Remarks 3.4 and 3.5, we have the following example.

\begin{example}
(1)	Let $X=SC_8^{2,4}=(x_0, x_1, x_2, x_3)$, where $x_0=(2m_1, 2m_2), x_1=(2m_1+1,2m_2-1), x_2=(2m_1+2,2m_2), x_3=(2m_1+1,2m_2+1)$, where $m_1, m_2 \in {\mathbb Z}$.
	Let us now consider the Cartesian product $X^2$.
	According to the condition of (3.1), the $G_{k^\ast}$-adjacency on $X^2$ is clearly obtained, where $k^\ast=k(2,4)=32$ (see (2.1)) so that $G_{k^\ast}=k^\ast=C_{k^\ast}$.\\
	Then, for each $p \in X^2$, we obtain $N_{G_{k^\ast}}(p,1)=N_{k^\ast}(p, 1)$ under the digital space $(X^2, G_{k^\ast})$, where $k^\ast=k(2,4)=32$.\\
	(2) We obtain $(SC_{26}^{3,5}\times SC_{26}^{3,5}, G_{k^\ast})$, where 
	$k^\ast=k(3,6)=232$, where $SC_{26}^{3,5}$ is shown in Figure 1(1) of \cite{H6}, i.e., $SC_{26}^{3,5}$ is assumed to be the set as an example in Figure 1(1).\\	
	In this case, 
	for each $p \in SC_{26}^{3,5}\times SC_{26}^{3,5}$, we obtain  $G_{k^\ast}=k^\ast$ so that $N_{G_{k^\ast}}(p,1)=N_{k^\ast}(p, 1)$ with the digital space $(SC_{26}^{3,5}\times SC_{26}^{3,5}, G_{k^\ast})$, where $k^\ast=k(3,6)=232$.\\ 
\end{example}

The $G_{k^\ast}$-adjacency of $X\times X$ of Definition 1 indeed plays a crucial role in studying a $DT$-$k$-group,  a $DT$-$k$-ring, and a $DT$-$k$-field and their substructures in Sections 4 and 5 (see Definitions 4, 5, 6, 7, and 8).\\

To represent a Cartesian product of $SC_{k} ^{n, l}$,
 as a matrix, we will use the notation\\
$$SC_{k} ^{n, l}:=(a_i)_{i \in [0, l-1]_{\mathbb Z}}.$$\\
 Then, take the Cartesian product using the following matrix:
 $$ SC_{k} ^{n, l} \times SC_{k} ^{n, l}:=(c_{i,\,j})_{(i,j) \in [0, l-1]_{\mathbb Z} \times [0, l-1]_{\mathbb Z}},  \eqno(3.7)$$
 where $c_{i,j}:=(a_i, a_j)$.\\

Indeed, the paper focuses on developing some algebraic structures in a digital-topological setting  such as $DT$-$k$-group, $DT$-$k$-ring, and $DT$-$k$-field (see Sections 4 and 5) endowed with $(X, k)$ and some algebraic structures on $X$ such as 
a group $(X, \ast_1)$ or a ring or a field $(X, \ast_1, \star)$.
Thus we will press on Definition 1 to establish a $DT$-$k$-ring and a $DT$-$k$-field which will be formulated in Sections 5 and 6.

As a special case of Lemma 4.6 of \cite{H7}, the following is obtained.

\begin{lemma}\cite{H7} For $SC_k^{n, l}$,
there is always a $G_{k^\ast}$-adjacency of $SC_{k}^{n, l} \times SC_{k}^{n, l}$, where $k^\ast:=k(t, 2n)$ in which the number $t$ is equal to number $t$ in $k:=k(t, n)$.
\end{lemma}

By Definition 2,  the following is obtained \cite{H7}.
\begin{example} (1) Assume $X:=[0, l]_{\mathbb Z}\times [0, l]_{\mathbb Z} \subset {\mathbb Z}^2$.
Then, a $G_{k^\ast}$-adjacency on $X \subset {\mathbb Z}^2$ is taken such that $k^\ast:=k(1, 2)=4$ (see Section 3 of \cite{H7}). Besides, for $p \in X$, $N_{G_{k^\ast}}(p, 1)\subset X$ is equal to $N_4(p, 1)\subset X$ (see (2.3) and (3.4)).\\
(2)  In $SC_8^{2,6} \times SC_8^{2,6}$ (see Remark 3.5), we obtain
$$N_{G_{32}}(c_{2,2}, 1) \neq  N_{32}(c_{2,2}, 1),$$
because
$$\sharp N_{G_{32}}(c_{2,2}, 1) = 5\,\,\text{and}\,\, \sharp N_{32}(c_{2,2}, 1) = 6. \eqno(3.8)$$
\end{example}

By Remark 3.3 and the property of (3.5), we obtain the following:
\begin{remark}  With the hypothesis stated in Definition 2 and (3.5), for $(X, k:=(t, n))$ we obtain the following:\\
(1) $N_{G_{k^\ast}}(p, 1)$ always exists in $X^2$, where  $k^\ast:=k(t, 2n)$ in which the number $t$ is equal to the number $t$ of $(X, k:=(t, n))$\cite{H7}.\\
(2) $N_{G_{k^\ast}}(p)$ need not be equal to $N_{k^\ast}^\ast(p)$, where
        $N_{k^\ast}^\ast(p):=\{q \in X^2\,\vert \, q \,\,\text{is}\,\,k^\ast\text{-adjacent to}\,\,p\}$, where the number $k^\ast$ is the digital connectivity of $X^2$ stated in (2.1). \\
(3) Not every $N_{k^\ast}(p, 1)$ in $X^2$ (see (2.3)) is always equal to $N_{G_{k^\ast}}(p,1), p \in X^2$.\\
            (4)  Given $(X, k:=k(t, n))$, assume an $N_{G_{k^\ast}}(p, 1)$, where $p:=(x_1, x_2) \in  X^2$.
Then, according to (3.5), we have the following equality \cite{H7}
 $$\sharp N_{G_{k^\ast}}(p, 1) = \sharp N_{k}(x_1, 1) + \sharp N_{k}(x_2, 1)-1.$$
\end{remark}

As a special case of Definition 3.8 and 4.18 of \cite{H7}, 
let us introduce the $(G_{k^\ast}, k)$-continuity of a map $f: (X^2, G_{k^\ast}) \to (X, k)$.

\begin{definition} \cite{H7} Given $(X, k)$, $X \subset {\mathbb Z}^{n}$,
consider the digital space $(X^2, G_{k^\ast})$. A function $f: (X^2, G_{k^\ast}) \to (X, k)$ is $(G_{k^\ast}, k)$-continuous at
a point $p:=(x_1, x_2)$ if for any point $q \in X^2$ such that $q \in N_{G_{k^\ast}}(p,1)$, we obtain
 $f(q) \in N_{k}(f(p), 1)$.
 In the case that the map $f$ is $(G_{k^\ast}, k)$-continuous at each point $p \in X^2$, we say that the map $f$ is $(G_{k^\ast}, k)$-continuous.
\end{definition}

Owing to Remark 3.5, compared to Proposition 2.1, we strongly note that the $(G_{k^\ast}, k)$-continuity of Definition 3 need not equal to any digital continuity of Proposition 2.1.
Besides, this $(G_{k^\ast}, k)$-continuity can be represented by using both a $G_{k^\ast}$-neighborhood and a digital $k$-neighborhood.
Namely, as a special case of Proposition 4.19 of \cite{H7}, we obtain the following (compare Propositions 2.1 and 3.10):

\begin{proposition}\cite{H7} Assume $(X^2, G_{k^\ast})$ induced by $(X, k)$.
 A map $f:(X^2, G_{k^\ast}) \to (X, k)$ is 
  $(G_{k^\ast}, k)$-continuous at the point $p \in X^2$ if and only if
 $$f(N_{G_{k^\ast}}(p, 1)) \subset N_{k}(f(p), 1).\eqno(3.9)$$
 A map $f:(X^2, G_{k^\ast}) \to (X, k)$ is  $(G_{k^\ast}, k)$-continuous if and only if
 for every point $p \in X^2$ we have $$f(N_{G_{k^\ast}}(p, 1)) \subset N_{k}(f(p), 1).$$
 \end{proposition}
 Indeed, this Proposition 3.10 will be essential for establishing a $DT$-$k$-group on $X$ (see Section 4).

\begin{corollary}  \cite{H7} Let $(X, 2n)$ be a $2n$-connected subset of $({\mathbb Z}^n, 2n)$. Then
 each of the projection maps $P_i: (X \times X, G_{4n}) \to (X, 2n)$ is a $(G_{4n}, 2n)$-continuous map, $i \in \{1, 2\}$, where $G_{4n}=4n$.
\end{corollary}

The $(G_{k^\ast}, k)$-continuity of Proposition 3.10 and Corollary 3.11 will facilitate a certain continuity of a multiplication for formulating a $DT$-$k$-group (see Definition 4) in Section 4.\\

Based on Propositions 2.1 and 3.10 and Definition 3, let us compare the  $(G_{k^\ast}, k)$-continuity and the typical $k$-continuity.

\begin{theorem} Assume a map $f:(X^2, G_{k^\ast}) \to (X, k)$. 
	Then for any point $(x, x^\prime) \in X^2$, the $(G_{k^\ast}, k)$-continuity of the restriction of the map $f$ onto $X\times \{x^\prime\}$ or $\{x\}\times X$, say
	$f\vert_{X\times \{x^\prime\}}: (X\times \{x^\prime\}, G_{k^\ast}) \to (X,k)$ or 	$f\vert_{\{x\}\times X}: (\{x\}\times X, G_{k^\ast}) \to (X,k)$,
		 is equivalent to the $k$-continuity of a self-map $g$ of $(X,k)$ for each $x$ and $x^\prime$.
			 \end{theorem}
\begin{proof} Given a relation set $(X^2, G_{k^\ast})$, for any point $(x, x^\prime) \in X^2$, each of the digital subspaces $(\{x\} \times X,  G_{k^\ast})$ and 
$(X \times \{x^\prime\}, G_{k^\ast})$ is equivalent to $(X, k)$ (see (3.1)).
Thus, by Propositions 2.1 and 3.10, the proof is completed.
	\end{proof}

\section{\bf A development of a $DT$-$k$-ring $(X, k, \ast_1, \star)$}\label{s6}

As a more luxurious structure than a $DT$-$k$-group, this section initially establishes  a $DT$-$k$-ring structure
derived from both $(X, k)$ and
 a ring  $(X, \ast_1, \star)$, denoted by $(X, k, \ast_1, \star)$.
 To do this work, first we need to recall the $DT$-$k$-group structure in \cite{H7} derived from both $(X, k)$ and a group $(X, \ast_1)$.\\

Based on the $G_{k^\ast}$-adjacency of $X\times X$ of Definition 1 and the $(G_{k^\ast}, k)$-continuity of Proposition 3.10, the notion of a $DT$-$k$-group was defined, as follows:
\begin{definition} \cite{H7}
  A $DT$-$k$-group, denoted by $(X, k, \ast_1)$, is $(X, k:=k(t, n))$ endowed with a group structure on $X \subset {\mathbb Z}^n$ using a certain binary operation $\ast_1$  such that\\
  (1) for any $(x, y) \in X^2$, the multiplication with respect to the operation $\ast_1$
$$\alpha: (X^2, G_{k^\ast}) \to (X, k)\,\,\text{defined by}\,\,\alpha(x, y)= x \ast_1 y\,\,\text{is}\,\,(G_{k^\ast}, k)\text{-continuous} \eqno(4.1)$$
and\\
(2) for each $x \in X$, the inverse map  with respect to the operation $\ast_1$
$$\beta: (X,k) \to (X,k)\,\,\,\text{defined by}\,\,\beta(x)= x^{-1}\,\,\text{is}\,\,k\text{-continuous},\,\,\,\,\,\,\,\,\,\eqno(4.2)$$
where $x^{-1}$ means the inverse element of $x$ with respect to the operation $\ast_1$.
 \end{definition}
   In particular, in the case that a group  $(X, \ast_1)$ is abelian, the $DT$-$k$-group $(X, k, \ast_1)$ is called abelian \cite{H7}. Besides, we recall that $SC_k^{n,l}$ is a sequence $(x_i)_{i \in [0,l-1]_{\mathbb Z}}$, which implies that there exists a circular order in the sequence, i.e., an existence of a linear order such as $x_i \lneq x_{i-1}, i \in [0,l-2]_{\mathbb Z}$.

 \begin{theorem} \cite{H7} (1) $(SC_k^{n,l}:=(x_i)_{i \in [0, l-1]_{\mathbb Z}}, \ast_1)$ is an abelian $DT$-$k$-group.\\
 	(2) $({\mathbb Z}^n, 2n, +)$ is an abelian $DT$-$2n$-group.
 \end{theorem}

As for Theorem 4.1, we need to recall that the element $x_0$ is the identity element of the group $(SC_k^{n,l}, \ast_1)$ under the operation $\ast_1$. 
Since Theorem 4.1 is strongly related to the development of some algebraic and topological structures that are richer than a $DT$-$k$-group,
we need to recall the process of establishing a $DT$-$k$-group structure of $(SC_k^{n,l}, \ast_1)$ \cite{H7} more precisely, as follows:\\
A group structure on $SC_k^{n, l}:=(x_i)_{i \in [0, l-1]_{\mathbb Z}}$ can be considered  in terms of the following operation $\ast_1$.
$$\ast_1:SC_k^{n, l} \times SC_k^{n, l} \to SC_k^{n, l}$$
 defined by
 $$\ast_1(x_i, x_j)=x_i \ast_1 x_j=x_{i+j(mod\,l)}. \eqno (4.3)$$
  Then it is clear that the group $(SC_k^{n, l}, \ast_1)$ is abelian \cite{H7}.\\

Unlike the operation of (4.3), let us assume another binary operation on $SC_k^{n, l}$.
 Namely, for two elements $x_i, x_j \in SC_k^{n,l}:=(x_i)_{i \in [0, l-1]_{\mathbb Z}}$,  we now  establish a binary operation on $SC_k^{n, l}$ \\
$$\star:SC_k^{n, l} \times SC_k^{n, l} \to SC_k^{n, l}$$
 defined by
 $$\star(x_i, x_j)=x_i \star x_j=x_{i\cdot j(mod\,l)}, \eqno (4.4)$$
 where we note that $\lq\lq$$i\cdot j(mod\,l)"$ means the typical multiplication of $i$ and $j$ under the operation $\lq\lq$$\cdot(mod\,l)"$.
Hereinafter, as for the operations (4.3) and (4.4), we will follow modular arithmetic.
Hence we obviously obtain the following:

\begin{lemma} Given $SC_k^{n,l}:=(x_i)_{i \in [0, l-1]_{\mathbb Z}}$,
$(SC_k^{n,l}, \star)$ is a semigroup with the unique unity $x_1$.
\end{lemma}
\begin{proof} It is clearly that the operation $\lq\lq\star$" of (4.4) is well-defined on $SC_k^{n, l}$ as a binary operation
and further, it is associative for formulating a semigroup on $SC_k^{n, l}$ under the operation $\lq\lq\star$".
Besides, it is clear that the element $x_1$ is the unique unity, i.e., for any elements $x_i (\neq x_0)\in SC_k^{n,l}$,
 we have
$$ x_i \star x_1=x_i=x_1 \star x_i. \eqno(4.5)$$
 \end{proof}

In view  of the given two operations $\lq\lq\ast_1$" and $\lq\lq\star$" in (4.3) and (4.4),
we can observe that the former is similar to an additive operation and the latter is also similar to  a multiplicative operation from the viewpoint of topological group theory.\\

Using the $(G_{k^\ast}, k)$-continuity, we now define the following.

\begin{definition}
  A $DT$-$k$-ring, denoted by $(X, k, \ast_1, \star)$, is an abelian $DT$-$k$-group $(X, k:=k(t, n), \ast)$ together with a ring structure on $X \subset {\mathbb Z}^n$, i.e., $(X, \ast_1, \star)$
     using the so-called additive binary operation $\ast_1$ and multiplicative operation $star$. 
     Namely,  a $DT$-$k$-ring, denoted by $(X, k, \ast_1, \star)$, is an abelian $DT$-$k$-group $(X, k:=k(t, n), \ast_1)$ combined with a ring structure on $X \subset {\mathbb Z}^n$, i.e., $(X, \ast_1, \star)$ such that for any $(x, y) \in X^2$ the following properties (1)--(3) hold.\\
  (1) For any element $(x,y) \in X^2$, the multiplication with respect to $\ast_1$
$$\alpha: (X^2, G_{k^\ast}) \to (X, k)\,\,\text{defined by}\,\,\alpha(x, y)= x \ast_1 y\,\,\text{is}\,\,(G_{k^\ast}, k)\text{-continuous}.$$
(2) For each $x \in X$,  the inverse map with respect to $\ast_1$
$$\beta: (X,k) \to (X,k)\,\,\text{defined by}\,\,\beta(x)= x^{-1}\,\,\text{is}\,\,k\text{-continuous.}$$ And\\
(3) for any $(x,y) \in X^2$, the multiplication with respect to the operation $\star$
$$ \gamma: (X^2, G_{k^\ast}) \to (X, k)\,\,\text{defined by}\,\,\gamma(x, y)= x \star y\,\,\text{is}\,\,(G_{k^\ast}, k)\text{-continuous}. \eqno(4.6)$$
 \end{definition}

\begin{remark} In Definition 5, in the case that a ring $(X, \ast_1, \star)$ is commutative, the $DT$-$k$-ring is said to be commutative.
\end{remark}

\begin{theorem}  Given $SC_k^{n,l}:=(x_i)_{i \in [0, l-1]_{\mathbb Z}}$,
$(SC_k^{n,l}, \ast_1, \star)$ is a commutative ring with unity $x_1$.
\end{theorem}
\begin{proof} (1) By Theorem 4.1, it turns out that $(SC_k^{n,l}, \ast_1)$ is an abelian group with identity $x_0$.\\
(2) By Lemma 4.2, $(SC_k^{n,l}, \star)$ is proved to be a commutative semigroup with unity $x_1$.\\
(3) Using the operations of (4.3) and (4.4), let us now prove the distributive law under the operations $\ast_1$ and $\star$.
To be specific, for any elements $x_i, x_j$, and $x_m$ in $SC_k^{n,l}$,
we obtain
$$\left \{
\aligned  & x_i\star (x_j \ast_1 x_m)=x_i\star x_{j+m (mod\,l)}= x_{(i\cdot(j+m)(mod\,l))(mod\,l)}\\
          &=x_{i\cdot j(mod\,l) + i\cdot m(mod\,l)}=(x_i\star x_j) \ast_1 (x_i\star x_m).
          \endaligned
\right\} \eqno(4.7)
$$
Similarly, we obtain
$$(x_i \ast_1 x_j)\star x_m= (x_i\star x_m)\ast_1 (x_j\star x_m),$$
which implies that $(SC_k^{n,l}, \ast_1, \star)$ is a ring. \\
(4) It is clear that $(SC_k^{n,l}, \star)$ is a commutative ring with unity $x_1$.\\
In view of (1)--(4) above, the proof is completed.
\end{proof}

To study a $DT$-$k$-ring, let us recall the typical ring $({\mathbb Z}_l:={\mathbb Z}/l{\mathbb Z}, +, \cdot)$ \cite{Fr1} under the arithmetic of both addition modulo $l$ and multiplication modulo $l$.
Then, based on the operations proposed in (4.3) and (4.4), we obtain the following:

 \begin{theorem} The commutative ring $(SC_k^{n,l}:=(x_i)_{i \in [0, l-1]_{\mathbb Z}}, \ast_1, \star)$ is isomorphic to the typical finite ring $({\mathbb Z}_l, +, \cdot)$,
 where the operations $\lq\lq$$+"$ and $\lq\lq$$\cdot"$ means the operations $\lq\lq$$+(\text{mod}\, l)"$ and $\lq\lq$$\cdot(\text{mod}\, l)"$, respectively.
  \end{theorem}
\begin{proof}
 First  of all, in terms of the operations  (4.3) and (4.4), by Theorem 4.1 and Lemma 4.2, it is clear that
  $(SC_k^{n,l}:=(x_i)_{i \in [0, l-1]_{\mathbb Z}}, \ast_1, \star)$ is a commutative ring with unity $x_1$.\\
Next, to prove an isomorphism between  $(SC_k^{n,l}, \ast_1, \star)$  and $({\mathbb Z}_l, +, \cdot)$, let us consider the two maps
    $$\left \{\aligned
    & \phi: (SC_k^{n,l}, \ast_1) \to  ({\mathbb Z}_l, +)\,\,\text{defined by} \,\,\phi(x_i)=i(\text{mod}\,l)\\
    & \text{under the operation of}\,\,(4.3)\,\,\text{and}\,\,\lq\lq +:=+(\text{mod}\,l)" \text{and}\\
    &  \psi: (SC_k^{n,l}, \star) \to  ({\mathbb Z}_l, \cdot)\,\,\text{defined by} \,\,\phi(x_i)=i(\text{mod}\,l)\\
    & \text{under the operation of} \,\,(4.4)\,\,\text{and}\,\,\lq\lq\cdot:=\cdot(\text{mod}\,l)".
\endaligned\right\} \eqno(4.8)
$$
 Then, each of the maps $\phi$ and $\psi$ is obviously a homomorphic bijection. 
 \end{proof}

Compared to Theorem 4.1, the following is obtained.

\begin{remark}  $({\mathbb Z}^n, 2n, +, \cdot)$ is not a $DT$-$2n$-ring.
\end{remark}
{\em Proof:} Although $({\mathbb Z}^n, +, \cdot)$ is obviously a ring and $({\mathbb Z}^n, 2n, +)$ is an abelian $DT$-$2n$-group (see Theorem 4.1),  $({\mathbb Z}^n, 2n, +, \cdot)$ cannot be a $DT$-$2n$-ring.
More precisely, let us check the condition (3) of Definition 5. 
Namely, we need to examine if the multiplication
$$\cdot:{\mathbb Z}^n \times {\mathbb Z}^n \to {\mathbb Z}^n \,\,\text{defined by} \,\,(x,y)\to x\cdot y  \eqno(4.9)$$
is $(G_{2n}, 2n)$-continuous.
After checking this process, we find that the multiplication $\lq\lq$$\cdot$" is not $(G_{2n}, 2n)$-continuous.\\
To be specific, in the case that $n=1$, let us examine if $({\mathbb Z}^n, 2n, +, \cdot)$ is a $DT$-$2n$-ring.
Namely, we need to prove the $(G_{4}, 2)$-continuity of the multiplication (see the condition (3) of Definition 5)
$$\cdot: ({\mathbb Z}^2, G_4)=({\mathbb Z}^2, 4) \to ({\mathbb Z}, 2)$$ in (4.9) (see also (4.6)).
Consider the two points $P=(2, 2)$ and $Q=(2, 3)$ so that $Q \in N_{G_4}(P, 1)= N_4(P, 1)$.
Then the images of $P$ and $Q$ by the multiplicative operation $\lq\lq$$\cdot$" are not $2$-adjacent in $({\mathbb Z}, 2)$ because
$$ (2, 2)\to 2\cdot 2=4 \,\,\text{and}\,\,(2, 3)\to 2\cdot 3= 6,$$
i.e., $4$ is not $2$-adjacent to $6$ in $({\mathbb Z}, 2)$.\\
In the case of $n \geq 2$, using an approach similar to the proof of the case of $n=1$, we can clearly prove that $({\mathbb Z}^n, 2n, +, \cdot)$ is not a $DT$-$2n$-ring, which completes the proof. $\Box$\\

Let us now examine if the ring $(SC_k^{n,l}, \ast_1, \star)$ is a $DT$-$k$-ring.

\begin{remark} $(SC_k^{n,l}, \ast_1, \star)$ is not a $DT$-$k$-ring.
 \end{remark}
 {\em Proof:}
Although $(SC_k^{n,l}, \ast_1, \star)$ is a ring,
the multiplication (see the condition (3) of Definition 5)
$$\gamma:(SC_k^{n,l} \times SC_k^{n,l}, G_{k^\ast}) \to SC_k^{n,l} $$ defined by
$\gamma (x_i, x_j)=x_i \star x_j=x_{i\cdot j(mod\,l)}$ is not $(G_{k^\ast}, k)$-continuous, which implies that  $(SC_k^{n,l}, \ast_1, \star)$ is not a $DT$-$k$-ring.
For instance, for $SC_8^{2,9}:=(x)_{i \in [0, 8]_{\mathbb Z}}$, consider the map
$$ \gamma: SC_8^{2,9} \times SC_8^{2,9} \to SC_8^{2,9} $$
related to the condition (3) of Definition 5.
Then, take the two elements $P:=(x_3, x_3)$ and $Q:=(x_3, x_4)$ so that $Q \in N_{G_{k^\ast}}(P, 1)$ (see the property of (3.1)), where $k^\ast=k(2,4)=32$.
After taking
$$\left \{
\aligned  & \gamma(P)=\gamma(x_3, x_3)=x_{9(mod\,9)}=x_0\,\,\text{and}\\
          & \gamma(Q)=\gamma(x_3, x_4)=x_{12(mod\,9)}=x_3,
          \endaligned
\right\} \eqno(4.10)
$$
so that $\gamma(P)$ is not $8$-adjacent to $\gamma(Q)$, which implies that  $(SC_8^{2,9}, \ast_1, \star)$ is not a $DT$-$8$-ring because the multiplication $\gamma$ with respect to the operations
$\ast_1$ and $\star$ does not satisfy the property of (4.6). $\Box$

\begin{corollary} For $l\in \mathcal{P}\setminus \{2,3\}$,  $(SC_k^{n,l}, \ast_1, \star)$ is not a $DT$-$k$-ring, where $\mathcal{P}$ denotes the set of prime numbers.
 \end{corollary}

\begin{example} Neither of $(SC_8^{2,4}, \ast_1, \star)$, $(SC_{26}^{3,5}, \ast_1, \star)$, $(SC_{18}^{3,7}, \ast, \star)$, and $(SC_{8}^{2,11}, \ast_1, \star)$ is a $DT$-$8$-ring, a $DT$-$26$-ring, a $DT$-$18$-ring, and a $DT$-$8$-ring, respectively.\\
	In detail, let us explain not to be a non-$DT$-$8$-ring of $(SC_8^{2,4}, \ast_1, \star)$.
	For convenience, let $SC_8^{2,4}:=(x_0, x_1, x_2,x_3)$.
	Then this ring $(SC_8^{2,4}, \ast_1, \star)$ does not satisfy the condition (3) of Definition 5.
	To be specific, while $(x_2,x_2) \in SC_8^{2,4} \times SC_8^{2,4}$ is $G_{k^\ast}$-adjacenct to $(x_2,x_3)$, where $k^\ast=32$, 
	the map
	$$\left \{
	\aligned  & \gamma: (SC_8^{2,4} \times SC_8^{2,4}, G_{32}) \to SC_8^{2,4},\\
	&\text{defined by}\,\,\gamma(x, y)= x \star y\,\,\text{is not}\,\,(G_{32}, 8)\text{-continuous}  \,\text(see (4.6)),
	\endaligned
	\right\}
	$$	 
because $\gamma((x_2,x_2))=x_0$ and $\gamma((x_2,x_3))=x_2$ so that $x_0$ is not $8$-adjacent to $x_0$.
\end{example}

  \begin{figure}[hbt]
     \begin{center}
\epsfig {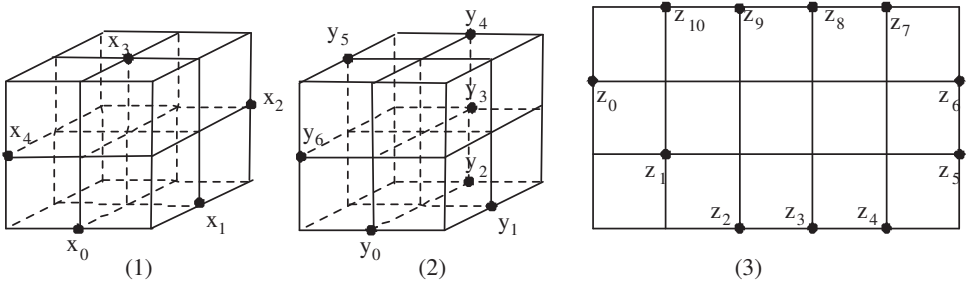}
\caption{For $k \in \{26, 18, 8\}$, configuration of non-being of a $DT$-$k$-ring structure of $(SC_{26}^{3,5}:=(x_i)_{i \in [0, 4]_{\mathbb Z}}, \ast_1, \star)$ (see Figure 1(c) of \cite{H6}), where $SC_{26}^{3,5}$ is $26$-isomorphic to the set
	$(X, 26)$ and $X:=\{x_0=(0,0,0), x_1=(-1,-1,1),x_2=(-1,0,2), x_3=(0, 1, 2),x_4=(1,1,1)\}$,
$(SC_{18}^{3,7}:=(y_i)_{i \in [0, 6]_{\mathbb Z}}, \ast_1, \star)$ (see (2)), and  $(SC_{8}^{2,11}:=(z_i)_{i \in [0, 10]_{\mathbb Z}}, \ast_1, \star)$ (see (3)), respectively.
Besides, each of $(SC_{26}^{3,5}:=(x_i)_{i \in [0, 4]_{\mathbb Z}}, \ast_1, \star)$, $(SC_{18}^{3,7}:=(y_i)_{i \in [0, 6]_{\mathbb Z}}, \ast_1, \star)$, and
$(SC_{8}^{2,11}:=(z_i)_{i \in [0, 10]_{\mathbb Z}}, \ast_1, \star)$ is a field (see Section 5) instead of a $DT$-$k$-ring.} \label{fig 1}
\end{center}
\end{figure}

\begin{definition} A subring of a  $DT$-$k$-ring is said to be a subset of the $DT$-$k$-ring which is a $DT$-$k$-ring under the induced operations from the given $DT$-$k$-ring.
In the case that a subring $(Y, k, \ast_1, \star)$ of $DT$-$k$-ring $(X, k, \ast_1, \star)$ is $k$-connected, we call it a $k$-connected $DT$-$k$-subring of $(X, k, \ast_1, \star)$.
\end{definition}

\section{\bf A $DT$-$k$-field structure derived from a digital object}\label{s4}

Let us now propose the notion of a $DT$-$k$-field.
 \begin{definition}
  A $DT$-$k$-field ($DT$-$k$-field for brevity), denoted by $(X, k, \ast_1, \star)$, is a commutative $DT$-$k$-ring $(X, k, \ast_1, \star)$ together with a field structure on $X \subset {\mathbb Z}^n$,
  i.e., $(X,\ast_1, \star)$,
   using the so-called additive binary operation $\ast_1$ and multiplicative operation $\star$ satisfying the conditions (1)--(3) of Definition 5.
     Furthermore, let $x_0$ be the identity under the additive operation $\ast_1$ in the $DT$-$k$-ring $(X, k, \ast_1, \star)$.
  Then, \\
  (4) for each $x \in X^\prime:=X \setminus \{x_0\}$, the inverse map with respect to the operation $\star$
$$\delta:  (X^\prime, k) \to (X^\prime, k)\,\,\text{given by}\,\,\delta(x)= x^{-1}\,\,\text{is}\,\,k\text{-continuous},\eqno(5.1)$$
where $x^{-1}$ means the inverse element of $x$ with respect to the operation $\star$.
 \end{definition}

\begin{remark}
 (1) A $DT$-$k$-field $(X, k, \ast_1, \star)$ is a commutative $DT$-$k$-ring $(X, k, \ast_1, \star)$.\\
  (2) A $DT$-$k$-ring $(X, k, \ast_1, \star)$ is an abelian $DT$-$k$-group $(X, k, \ast_1)$.\\
  However, the converse of each of (1) and (2) does not hold.
  \end{remark}

Let us now investigate some examples for $DT$-$k$-fields.

\begin{lemma}  Let $X:=\{-1, 0, 1\}$, after defining the (additive) operation $\ast_1$ on $X$, the set $(X, \ast_1)$ is an abelian group.
\end{lemma}
\begin{proof} Consider the so-called additive operation $\ast_1$ on $X$ given by
$$\ast_1: X \times X \to X\,\,\text{defined by}\,\,\ast_1(x, x^\prime)=(x+x^\prime)(mod\,2).$$
  Namely, we have $$\left \{\aligned
    &  (-1)\ast_1(-1)=0, (-1)\ast_1(1)=0= (1)\ast_1(-1),\\
    & 1\ast 1=0, x\ast_1 0=x=0\ast_1 x, \,\,\text{for any}\,\,x \in X.
\endaligned\right\} \eqno(5.2)
$$
Then it is obvious that the pair $(X, \ast_1)$ is an abelian group.
\end{proof}

\begin{lemma}  Let $X:=\{-1, 0, 1\}$.
After defining the (multiplicative) operation $\star$ on $X$, the set $(X, \star)$ is a commutative semigroup with unity $\lq\lq$1".
\end{lemma}
\begin{proof} Consider the so-called multiplicative operation defined by
 $$\star: X \times X \to X\,\,\text{defined by}\,\,\star(x, x^\prime)=(x\cdot x^\prime)(mod\,2).$$
 Namely,  
   $$\left \{\aligned
    &  (-1)\star (-1)=1, (-1)\star 1=-1= 1\star(-1),\\
    &  x\star 1=x=1\star x, \,\,\text{for any}\,\,x(\neq 0) \in X\,\,\text{and}\\
    &  x\star 0=0=0\star x, \,\,\text{for any}\,\,x \in X.
\endaligned\right\} \eqno(5.3)
$$
Then it is obvious that the pair $(X, \star)$ is a commutative semigroup with unity $\lq\lq$$1"$. 
\end{proof}

Using the two operations $\ast_1$ and $\star$ in Lemmas 5.2 and 5.3, we obtain the following:

\begin{theorem} Let $X=\{-1, 0, 1\}$ and consider $(X, 2)$.
After defining the (additive) operation $\ast_1$  on $X$ in Lemma 5.2 and (multiplicative) operation $\star$ on $X$ in Lemma 5.3, the set $(X, \ast_1, \star)$ is a $DT$-$2$-field.
\end{theorem}
{\em Proof:} Owing to Lemmas 5.2 and 5.3 and the property of (5.3), the set $(X, \ast_1, \star)$ is a field.
Let us now define the following three maps $\alpha, \beta$, and $\gamma$, i.e., \\
$(1)\, \alpha$ from $(X^2, G_4)$ onto $(X, 2)$ as in Definition 5(1) (or Definition 7(1)),\\
$(2)\, \beta$ as a self-map of $(X, 2)$ as in Definition 5(2) (or Definition 7(2)), and \\
 $(3)\, \gamma$ from $(X^2, G_4)$ onto $(X, 2)$ as in (4.6).\\
Furthermore, using the operations $\ast_1$ of (5.2) and $\star$ of (5.3),
we need to prove the $(G_{4}, 2)$-continuity of $\alpha$, the $2$-continuity of $\beta$, and the $(G_{4}, 2)$-continuity of $\gamma$, as follows:
Namely, we concretely obtain the following properties based on the (1)--(3) above:\\
$(1)^\prime$ For each $(x, y) \in X^2$, the multiplication with respect to $\ast_1$ of (5.2)
$$\alpha: (X^2, G_{4}) \to (X, 2)\,\,\text{defined by}\,\,\alpha(x, y)= x \ast_1 y\,\,\text{is}\,\,(G_{4}, 2)\text{-continuous},$$
$(2)^\prime$ for any $x \in X$, the inverse map with respect to the operation $\ast_1$ of (5.2)
$$\beta: (X, 2) \to (X, 2)\,\,\text{given by}\,\,\beta(x)= x^{-1}\,\,\text{is}\,\,2\text{-continuous, and}\,\,\,\,\,\,\,\,\,\,\,\,$$
$(3)^\prime$ for each $(x, y) \in X^2$, the multiplication with respect to the operation $\star$ of (5.3)
$$\,\gamma: (X^2, G_{4}) \to (X, 2)\,\,\text{defined by}\,\,\gamma(x, y)= x \star y\,\,\text{is}\,\,(G_{4}, 2)\text{-continuous (see Figure 2)}.$$
Finally, let us now examine if the inverse map under the operation $\star$ as a self-map of $(X^\prime, 2)$ is $2$-continuous, where $X^\prime=\{-1,1\}$, as follows:
 (4)  for any $x \in X^\prime$,
 $$\left \{\aligned
    & \delta:  (X^\prime, 2) \to (X^\prime, 2)\,\,\text{given by}\,\,\delta(x)= x^{-1}\,\,\text{under the operation}\,\,\star\\
    & \text{is clearly}\,\,2\text{-continuous (see Proposition 2.1)},
\endaligned\right\}
$$
where $x^{-1}$ means the inverse element of $x$ with respect to the operation $\star$.\\
Owing to the satisfaction of the four properties (1)--(4) above, we conclude that  $(X, 2, \ast_1, \star)$ is a $DT$-$2$-field. $\Box$\\

  \begin{figure}[hbt]
     \begin{center}
\epsfig {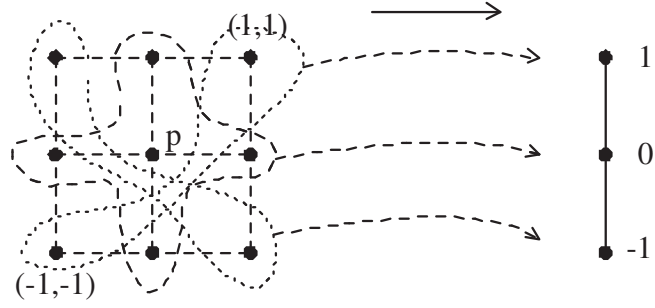}
\caption{Configuration of the $(G_4, 2)$-continuity of the multiplication $\gamma:(X^2, G_4=4) \to (X, 2)$ of $(3)^\prime$ in the proof of Theorem 5.4 for being a $DT$-$2$-field of $(X, 2, \ast_1, \star)$.
Let $p=(0,0)$. Then $\gamma(N_{G_4}(p,1))=\{0\}, \gamma(A)=1, A:=\{(-1,-1), (1,1)\}$, and $\gamma(B)=-1, B:=\{(-1,1), (1,-1)\}$ because $N_{G_4}(p,1)=N_4(p,1)$.} \label{fig 2}
\end{center}
\end{figure}

Let us now suggest another example for a $DT$-$2$-field.
\begin{theorem} Let $Y:=\{0, 1\}$ and consider $(Y, 2)$.
After defining the (additive) operation $\ast_1$  on $Y$ as in Lemma 5.2 and (multiplicative) operation $\star$ on $Y^\prime$ in Lemma 5.3, the set $(Y, 2, \ast_1, \star)$ is a $DT$-$2$-field.
\end{theorem}
{\em Proof:}
By using a method similar to the proof of Theorem 5.4, the proof is completed.
To be specific, owing to Lemmas 5.2 and 5.3, the set $(Y, \ast_1, \star)$ is a field.
Let us now define the maps $\alpha:(Y^2, G_4) \to (Y, 2)$ as in Definition 5(1) (or Definition 7(1)) and $\beta$ as a self-map of $(Y, 2)$ as in Definition 7(2), and  $\gamma:(Y^2, G_4) \to (Y, 2)$ as in Definition 5(3) (or Definition 7(3)).
Then, we need to prove the $(G_{4}, 2)$-continuity of $\alpha$, the $2$-continuity of $\beta$, the $(G_{4}, 2)$-continuity of $\gamma$, as follows:
Namely, we obtain the following:\\
(1) For each $(x, y) \in Y^2$, the multiplication with respect to $\ast_1$ of (5.2)
$$\,\alpha: (Y^2, G_{4}) \to (Y, 2)\,\,\text{given by}\,\,\alpha(x, y)= x \ast_1 y\,\,\text{is}\,\,(G_{4}, 2)\text{-continuous},$$
(2) for any $x \in Y$, the inverse map with respect to the operation $\ast_1$
$$\,\beta: (Y, 2) \to (Y, 2)\,\,\text{defined by}\,\,\beta(x)= x^{-1}\,\,\text{is}\,\,2\text{-continuous,}\,\,\,\,\,\,\,\,\,\,\,$$
(3) for each $(x, y) \in Y^2$,  the multiplication with respect to the operation $\star$ as in (5.3)
$$ \,\gamma: (Y^2, G_{4}) \to (Y, 2)\,\,\text{given by}\,\,\gamma(x, y)= x \star y\,\,\text{is}\,\,(G_{4}, 2)\text{-continuous}.$$
Finally, let us examine if the inverse map under the operation $\star$ as a self-map of $(Y^\prime, 2)$ is $2$-continuous, where $Y^\prime=\{1\}$, as follows:\\
(4) For any $x \in Y^\prime$,
$$\,\delta:  (Y^\prime, 2) \to (Y^\prime, 2)\,\,\text{given by}\,\,\delta(x)= x^{-1}\,\,\text{under the operation}\,\,\star\,\,\text{is clearly}\,\,2\text{-continuous},$$
where $x^{-1}$ means the inverse element of $x$ with respect to the operation $\star$.\\
Owing to the satisfaction of the four properties (1)--(4) above, we conclude that  $(Y, \ast_1, \star)$ is a $DT$-$2$-field.
$\Box$\\

Let us now define a subfield of a $DT$-$k$-field.
\begin{definition} A subfield of a  $DT$-$k$-field (or $DT$-$k$-subfield) is said to be a subset of the $DT$-$k$-field which is a $DT$-$k$-field under the induced operations from the given $DT$-$k$-field.
In the case that a subfield $(Y, k, \ast_1, \star)$ of $DT$-$k$-field $(X, k, \ast_1, \star)$ is $k$-connected, we call it a $k$-connected subfield of $(X, k, \ast_1, \star)$.
\end{definition}

\begin{example} The $DT$-$2$-field $(Y, 2, \ast_1, \star)$ in Theorem 5.5 is a $2$-connected subfield of the $DT$-$2$-field $(X, \ast_1, \star)$  of Theorem 5.4.
\end{example}

Based on the $DT$-$k$-group structure of $(SC_k^{n,l}, \ast_1)$ as in Theorem 4.1, let us now examine if $(SC_k^{n,l}, \ast_1, \star)$ is a  $DT$-$k$-field with the multiplicative operation $\star$ in (4.4).

\begin{lemma} For any $k$-adjacency of ${\mathbb Z}^n$ and any $l$ with $5 \leq l \in \mathcal{P}$, $(SC_k^{n,l}, \ast_1, \star)$ is a field.
\end{lemma}
{\em Proof:}
(1) $(SC_k^{n,l}, \ast_1)$ is an abelian group with the identity $x_0$ under the operation $\ast_1$ in (4.3).\\
(2) $(SC_k^{n,l}, \ast_1, \star)$ is a commutative ring with unity $x_1$ under the operation $\star$ in (4.4).\\
(3) Any element $x_i(\neq x_0) \in SC_k^{n,l}$ has a unique inverse element under the operation  $\star$. $\Box$

\begin{theorem} For $l\in \mathcal{P}\setminus \{2,3\}$, the field $(SC_k^{n,l}, \ast_1, \star)$ is isomorphic to $({\mathbb Z}_l, +, \cdot)$.
 \end{theorem}
 {\em Proof:}
  First  of all, owing to the operations of (4.3) and (4.4), it is clear that
  $(SC_k^{n,l}:=(x_i)_{i \in [0, l-1]_{\mathbb Z}}, \ast_1, \star)$ is a field with unity $x_1$.\\
Next, to prove a field isomorphism between  $(SC_k^{n,l}, \ast_1, \star)$  and $({\mathbb Z}_l, +, \cdot)$, let us consider the maps $\phi$ and $\psi$ in (4.8).   
 Then, each of $\phi$ and $\psi$ is a homomorphic bijection. $\Box$

\begin{remark} For any $k$-adjacency of ${\mathbb Z}^n$ and any $l$ with $4 \leq l \in \mathcal{P}\setminus \{2,3\}$,
$(SC_k^{n,l}, \ast_1, \star)$ is not a  $DT$-$k$-field because $(SC_k^{n,l}, \ast_1, \star)$ is not a  $DT$-$k$-ring (see Remark 4.6).
\end{remark}

\begin{theorem} (1) The $n$-dimensional digital cube $(X^n,\ast_1, \star)$ is a $DT$-$2n$-field, where $X:=\{-1, 0, 1\}$.\\
(2) $(Y^n,\ast, \star)$ is a $DT$-$2n$-field, where $Y:=\{0, 1\}$.
\end{theorem}
{\em Proof:}
(1) First of all, consider the $G_{k^\ast}$-adjacency of $X^{2n}$ derived from $(X^n, 2n)$, where $k^\ast=k(1,2n)=4n$.
Then we need to consider the maps $\alpha, \beta, \gamma$, and $\delta$ as stated in Definitions 5 and 7 such that each of them satisfies the corresponding continuity.
More precisely, for $p=(p_1, \cdots, p_n)$ and $q=(q_1, \cdots, q_n)$ in $ X^{n}$, we obtain the maps supporting  a $DT$-$2n$-field of $(X^n,\ast_1, \star)$ with the following operations $\ast_1$ and $\star$ on $X^n$
using the operations in (5.2) and (5.3).

$$\left \{\aligned
    & (1)\,\ast_1: X^n \times X^n \to X^n\,\, \text{given by}\\
& \ast_1(p,q)=p\ast_1 q=(p_1 \ast_1 q_1, \cdots, p_n \ast_1 q_n)\,\,\text{and}\\
& (2)\star: X^n \times X^n \to X^n\,\, \text{defined by}\\
&\star(p,q)=p\star q=(p_1 \star q_1, \cdots, p_n \star q_n).
   \endaligned\right\} \eqno(5.4)
$$
Using the operations in (5.4), we now prove the following:
$$\left \{\aligned
    & (1)\,\alpha: (X^n \times X^n, G_{k^\ast}=4n)  \to (X^n, 2n)\,\, \text{given by}\\
& \alpha(p,q)=p\ast_1 q=(p_1 \ast_1 q_1, \cdots, p_n \ast_1 q_n)\,\,\text{is}\,\,(G_{k^\ast}, 2n)\text{-continuous,}\\
&\text{where the operation}\,\, p_i \ast_1 q_i, i \in [1,n]_{\mathbb Z},\,\,\text{is that in (5.2)}.
   \endaligned\right\}
$$

$$\left \{\aligned
    & (2)\beta: (X^n, 2n)  \to (X^n, 2n)\,\,\text{given by}\\
&\beta(p)=(p_1^{-1}, \cdots, p_n^{-1})\,\, \text{is}\,\,2n\text{-continuous (see Proposition 2.1),}\\
& \text{where} \,\,p_i^{-1}\,\,\text{is the inverse element with respect to the operation}\,\,\ast_1\,\,\text{in (5.2)}.
   \endaligned\right\}
$$
   $$\left \{\aligned
    & (3) \text{Using the operation of (5.3),} \\
    & \gamma: (X^n \times X^n, G_{k^\ast}=4n)  \to (X^n, 2n)\,\, \text{defined by}\\
 & \gamma(p,q)=p\star q=(p_1 \star q_1, \cdots, p_n \star q_n)\,\,\text{is}\,\,(G_{k^\ast}, 2n)\text{-continuous.}
   \endaligned\right\}
   $$
Finally,
   $$\left \{\aligned
    & (4) \, \delta:  (X^n, 2n)  \to (X^n, 2n)\,\,\text{given by}\\
    &\delta(p)=(p_1^{-1}, \cdots, p_n^{-1})\,\, \text{is}\,\,2n\text{-continuous, where} \,\,p_i^{-1}, i \in [1,n]_{\mathbb Z},\\
    &\text{is the inverse element with respect to the operation}\,\,\star\,\,\text{in (5.3)}.
   \endaligned\right\}
$$
$\Box$

\begin{remark} A finite digital plane $(X, k), X\subset {\mathbb Z}^2$, need not be a $DT$-$k$-field.
\end{remark}

In view of Theorem 5.10 and Remark 5.11, it is clear that not every digital cube $(Y^n, 2n)$ is always a $DT$-$2n$-field.
However, the following is obtained.

\begin{remark} (1) Assume  $(W^n, 2n)$, where $W:=[a, a+2]_{\mathbb Z}$ and $n \in {\mathbb N}$. After transforming $(W^n, 2n)$ onto $(X^n, 2n)$ via a $2n$-isomorphism, where $X:=[-1, 1]_{\mathbb Z}$, 
	the transformed digital image is a $DT$-$2n$-field.\\
	(2) Assume  $(V^n, 2n)$, $V:=[a, a+1]_{\mathbb Z}$ and $n \in {\mathbb N}$.
	After transforming $(V^n, 2n)$ onto $(Y^n, 2n)$ in terms of a $2n$-isomorphism, where $Y:=[0, 1]_{\mathbb Z}$, 
	the transformed digital image is a $DT$-$2n$-field.
	\end{remark}  

\section{\bf Isomorphisms for $DT$-$k$-groups,  $DT$-$k$-rings, and  $DT$-$k$-fields,  and development of pseudo $DT$-$k$-rings and pseudo $DT$-$k$-fields}\label{s5}

This section initially introduces the notions of isomorphisms for $DT$-$k$-groups,  $DT$-$k$-rings, and  $DT$-$k$-fields.

\begin{definition} Given two $DT$-$k_i$-groups $(X, k_1, \ast_1)$ and $(Y, k_2, \ast_2)$, where $X \subset {\mathbb Z}^{n_1}$ and $Y \subset {\mathbb Z}^{n_2}$,
we say that the map  $h_G:(X, k_1, \ast_1) \to (Y, k_2, \ast_2)$ is a $DT$-$(k_1, k_2)$-group isomorphism if $h_G$ is a $DT$-$(k_1, k_2)$-isomorphic bijection with respect to the operations $\ast_1$ and $\ast_2$,
where we say that the map $h_G$ is a $DT$-$(k_1, k_2)$-isomorphic bijection with respect to the operations $\ast_1$ and $\ast_2$ if \\
(1) the map $h_G:(X, k_1) \to (Y, k_2)$ is a $(k_1, k_2)$-isomorphism and \\
(2) $h_G:(X, \ast_1) \to (Y, \ast_2)$ is a group homomorphism. \\
In the case that  $n_1 = n_2$ and $k_1 = k_2$, we say that it is a $DT$-$k_1$-group isomorphism.
\end{definition}

\begin{example} $(SC_k^{n,l}, \ast_1)$ is $DT$-$(k, 2)$-group isomorphic to $({\mathbb Z}_l, 2, +)$
\end{example}

\begin{theorem} $(SC_{k_1}^{n_1,l_1}, \ast_1)$ is  $DT$-$(k_1, k_2)$-group isomorphic to $(SC_{k_2}^{n_2,l_2}, \ast_1)$ if and only if $l_1=l_2$.
\end{theorem}
{\em Proof:} $(\Rightarrow)$: The proof is straightforward. \\
$(\Leftarrow)$ In the case of $l_1=l_2$, let
$$h: (SC_{k_1}^{n_1,l_1}:=(a_i)_{i \in [0, l_1-1]_{\mathbb Z}}, \ast_1) \to (SC_{k_2}^{n_2,l_2}:=(b_i)_{i \in [0, l_2-1]_{\mathbb Z}}, \ast_1)$$
given by
$$h(a_i)=b_i, i \in [0, l_1-1]_{\mathbb Z}.$$
Then it is clear that $h$ is a $DT$-$(k_1, k_2)$-group isomorphism. $\Box$

\begin{definition} Given two $DT$-$k_i$-rings $(X, k_1, \ast_1, \star_1)$ and $(Y, k_2, \ast_2, \star_2)$,
we say that the map  $h_R: (X, k_1, \ast_1, \star_1) \to (Y, k_2, \ast_2, \star_2)$ is a $DT$-$(k_1, k_2)$-ring isomorphism if $h_R$ is a $DT$-$(k_1, k_2)$-isomorphic bijection with respect to the
operations  $\ast_1$ and $\ast_2$ as well as $\star_1$ and $\star_2$, 
where we say that 
 that the map $h_R$ is a $DT$-$(k_1, k_2)$-isomorphic bijection with respect to the operations $\ast_1$ and $\ast_2$ as well as $\star_1$ and $\star_2$ if \\
(1) the map $h_R:(X, k_1) \to (Y, k_2)$ is a $(k_1, k_2)$-isomorphism, \\
(2) $h_R:(X, \ast_1) \to (Y, \ast_2)$ is a group homomorphism and\\
(3) $h_R:(X,  \star_1) \to (Y,  \star_2)$ is a homomorphism, i.e., for $x, x^\prime \in X$, $h_R(x \star_1 x^\prime)=h_R(x) \star_2 h_R(x^\prime)$.  \\
In the case that  $n_1 = n_2$ and $k_1 = k_2$, we say that it is a $DT$-$k_1$-ring isomorphism.
\end{definition}

\begin{definition} Given two $DT$-$k_i$-fields $(X, k_1, \ast_1, \star_1)$ and $(Y, k_2, \ast_2, \star_2)$,
we say that the map  $h_F: (X, k_1, \ast_1, \star_1) \to (Y, k_2, \ast_2, \star_2)$ is a $DT$-$(k_1, k_2)$-field isomorphism if $h_F$ is a $DT$-$(k_1, k_2)$-isomorphic bijection with respect to the
operations  $\ast_1$ and $\ast_2$ as well as $\star_1$ and $\star_2$, 
where we say that 
that the map $h_F$ is a $DT$-$(k_1, k_2)$-isomorphic bijection with respect to the operations $\ast_1$ and $\ast_2$ as well as $\star_1$ and $\star_2$ if \\
(1) the map $h_F:(X, k_1) \to (Y, k_2)$ is a $(k_1, k_2)$-isomorphism, \\
(2) $h_F:(X, \ast_1) \to (Y, \ast_2)$ is a group homomorphism and\\
(3) $h_F:(X,  \star_1) \to (Y,  \star_2)$ is a homomorphism, i.e., for $x, x^\prime \in X$, $h_F(x \star_1 x^\prime)=h_F(x) \star_2 h_F(x^\prime)$.  \\
In the case that  $n_1 = n_2$ and $k_1 = k_2$, we say that it is a $DT$-$k_1$-field isomorphism.
\end{definition}

Let us now introduce the notions of a pseudo $DT$-$k$-ring and a pseudo  $DT$-$k$-field that are weaker than those of a $DT$-$k$-ring and a $DT$-$k$-field, respectively.

\begin{definition} A  pseudo $DT$-$k$-ring, denoted by  $(X, k, \ast_1, \star)$, is an abelian  $DT$-$k$-group  $(X, k, \ast_1)$ combined with a ring structure on $X \subset {\mathbb Z}^n$ using the so-called multiplicative operation in $\star$ $(X, \ast_1, \star)$.
\end{definition}
After comparing Definitions 5 and 12, we can recognize that a pseudo $DT$-$k$-ring $(X, k, \ast_1, \star)$ is a  $DT$-$k$-group $(X,k, \ast_1)$ with a ring $(X, k, \ast_1, \star)$ so that
 a pseudo $DT$-$k$-ring of Definition 12 is weaker than a $DT$-$k$-ring because the condition (3) of Definition 7 is deleted in Definition 12.

\begin{example} $(SC_{k}^{n,l}, \ast_1, \star)$ is a pseudo $DT$-$k$-ring.
\end{example}

\begin{definition}  A pseudo $DT$-$k$-field, denoted by  $(X, k, \ast_1, \star)$, is a  
	$DT$-$k$-ring  $(X, k, \ast_1, \star)$ endowed with a field structure on $X \subset {\mathbb Z}^n$ using the so-called additive binary operation $\ast_1$ and the multiplicative operation $\star$.
\end{definition}

After comparing Definitions 5 and 13, we can recognize 
 a pseudo $DT$-$k$-ring is weaker than a $DT$-$k$-field because the condition (4) of Definition 8 is deleted in Definition 13.

\section{\bf Some remarks and further work}\label{s6}

The paper developed several algebraic structures based on digital objects such as a $DT$-$k$-ring, a $DT$-$k$-field, a $DT$-$k$-group isomorphism, a $DT$-$k$-ring isomorphism, a $DT$-$k$-field isomorphism, a pseudo-$DT$-$k$-ring, a pseudo-$DT$-$k$-field.
   Besides, the paper suggested several examples for a $DT$-$k$-field and classifying $DT$-$k$-groups, $DT$-$k$-rings, $DT$-$k$-fields with respect to a $DT$-$k$-group isomorphism, a $DT$-$k$-ring isomorphism, a $DT$-$k$-field isomorphism, respectively.\\
   As further work, we need to investigate some conditions for supporting the product property of  $DT$-$k$-groups. Besides, based on an adjacency from the $n$-dimensional Khalimsky topological structure \cite{HW1}, we need to develop some algebraic structures.

   \bigskip

{\bf Funding}: 
National Research Foundation of~Korea funded by the Ministry of Education, Science and Technology (2019R1I1A3A03059103)

 \bigskip

{\bf Conflicts of Interest}: The author declares no conflict of interest.

 \bigskip
 
{\bf Data availability}: The author confirms that the data supporting the findings of this study are available within this article.

\end{document}